\documentclass[11pt,reqno]{amsart}
\usepackage{caption}
 \usepackage{amssymb,amsfonts,amscd,amsbsy, color, amsmath}
\usepackage[mathscr]{eucal}
\usepackage{url}
\usepackage{graphicx}
\usepackage{mathtools}
\usepackage{todonotes}
\usepackage{tabularx}
\usepackage{tabu}
\usepackage{float, color}
\usepackage{placeins}
\usepackage{cite}

\newtheorem{theorem}{Theorem}[section]

\newtheorem{corollary}[theorem]{Corollary}

\newtheorem{conjecture}[theorem]{Conjecture}
\theoremstyle{definition}
\newtheorem{remark}[theorem]{Remark}
\numberwithin{equation}{section}

\makeatletter
\def\imod#1{\allowbreak\mkern5mu({\operator@font mod}\,\,#1)}
\makeatother
\allowdisplaybreaks

\makeatletter
\@namedef{subjclassname@2020}{%
  \textup{2020} Mathematics Subject Classification}
\makeatother

\begin{document}

\title[Asymptotics and sign patterns for coefficients in expansions of Habiro elements]{Asymptotics and sign patterns for coefficients in expansions of Habiro elements}

\author{Ankush Goswami}

\author{Abhash Kumar Jha}

\author{Byungchan Kim}

\author{Robert Osburn}

\address{School of Mathematical and Statistical Sciences, the University of Texas Rio Grande Valley, 1201 W. University Dr., Edinburg, TX, 78539}
\email{ankush.goswami@utrgv.edu, ankushgoswami3@gmail.com}

\address{Department of Mathematical Sciences, Indian Institute of Technology (Banaras Hindu University), Varanasi - 221005, India}
\email{abhash.mat@iitbhu.ac.in, abhashkumarjha@gmail.com}

\address{School of Natural Sciences, Seoul National University of Science and Technology, 232 Gongneung-ro, Nowon-gu, Seoul, 01811, Republic of Korea}
\email{bkim4@seoultech.ac.kr}

\address{School of Mathematics and Statistics, University College Dublin, Belfield, Dublin 4, Ireland}
\email{robert.osburn@ucd.ie}


\date{\today}

\begin{abstract}
We prove asymptotics and study sign patterns for coefficients in expansions of elements in the Habiro ring which satisfy a strange identity. As an application, we prove asymptotics and discuss positivity for the generalized Fishburn numbers which arise from the Kontsevich-Zagier series associated to the colored Jones polynomial for a family of torus knots. This extends Zagier's result on asymptotics for the Fishburn numbers.
\end{abstract}

\maketitle

\section{Introduction}
The expression
\begin{equation}\label{f}
F(q) := \sum_{n=0}^{\infty} (q)_n
\end{equation}
first occurred in a talk entitled  ``Analytic continuation of Feynman integrals" by Kontsevich as part of the Seminar on Algebra, Geometry and Physics at MPIM Bonn on October 14, 1997. Here and throughout,
\begin{equation*}
(a)_n = (a;q)_n := \prod_{k=1}^{n} (1-aq^{k-1})
\end{equation*}
is the standard $q$-hypergeometric notation. Note that (\ref{f}) does not converge on any open subset of $\mathbb{C}$, but is well-defined when $q$ is a root of unity (where it is finite) and when $q$ is replaced by $1-q$. Moreover, $F(q)$ is an element of the Habiro ring \cite{habiro}
\begin{eqnarray*}
\displaystyle \mathcal{H} := \varprojlim_{n} \mathbb{Z}[q] / \langle (q)_n \rangle. 
\end{eqnarray*}
Motivated by Kontsevich's lecture and Stoimenow's work on regular linearized chord diagrams \cite{sto}, Zagier determined the asymptotic behavior for the Fishburn numbers $\xi(n)$ which are the coefficients in the formal power series expansion

\begin{equation*} \label{fish}
F(1-q) =: \sum_{n=0}^{\infty} \xi(n) q^n = 1 + q + 2q^2 + 5q^3 + 15q^4 + 53q^5 + \cdots.
\end{equation*}
Namely, as $n \to \infty$ \cite[Theorem 4]{z1}

\begin{equation} \label{fishasy}
\xi(n)\sim \Bigl(\dfrac{6}{\pi^2}\Bigr)^n n! \sqrt{n} \left(C_0+\dfrac{C_1}{n}+\cdots\right)    
\end{equation}
where $C_0=\frac{12\sqrt{3}}{\pi^{\frac{5}{2}}}e^{\tfrac{\pi^2}{12}},\;C_1=C_0\left(\frac{3}{8}-\frac{17\pi^2}{144}+\frac{\pi^4}{432}\right)$ and all $C_i$ are effectively computable constants. A key step in proving (\ref{fishasy}) is the ``strange identity"

\begin{equation} \label{strid}
F(q) ``=" -\frac{1}{2} \sum_{n=1}^{\infty} n \Bigl( \frac{12}{n} \Bigr) q^{\frac{n^2 - 1}{24}}
\end{equation}
where $``="$ means that the two sides agree to all orders at every root of unity (for further details, see \cite[Sections 2 and 5]{z1}) and $\bigl( \frac{12}{*} \bigr)$ is the quadratic character of conductor $12$. The idea is to first express $\xi(n)$ in terms of the Taylor series coefficients of $F(e^{-t})$, then employ (\ref{strid}) to, ultimately, obtain estimates for these coefficients. These estimates, in turn, lead to (\ref{fishasy}). ``Identities" such as (\ref{strid}) are not only important in proving asymptotics, but also play a crucial role in obtaining congruences for $\xi(n)$ modulo prime powers \cite{akl} and quantum modularity for $F(q)$ \cite{z}. For developments in these latter two directions, see \cite{ak, akl, as, Be, garvan, an, ano, gkr, straub}. The positivity of the Fishburn numbers $\xi(n)$ is a consequence of any of its numerous combinatorial interpretations \cite[A022493]{oeis}. The purpose of this paper is to prove asymptotics and study sign patterns for coefficients in expansions of elements in $\mathcal{H}$ which satisfy a general type of strange identity. Before stating our main result, we introduce some notation.

Let $f:\mathbb{Z}\rightarrow\mathbb{C}$ be a function of period $M\geq 2$. For integers $a \geq 0$ and $b>0$, consider the partial theta series
\begin{equation*}
\theta_{a,b,f}^{(\nu)}(q):=\sum_{n=0}^{\infty} n^\nu f(n)q^{\frac{n^2-a}{b}}
\end{equation*}
where $\nu \in \{0, 1\}$. Suppose there exists
\begin{eqnarray*}\label{Habiro}
F_{f}(q) := \sum_{n=0}^\infty A_{n, f}(q)(q)_n \in \mathcal{H}
\end{eqnarray*}
where $A_{n, f}(q) \in \mathbb{Z}[q]$ such that 
\begin{eqnarray}\label{gen-strange}
F_f(q) ``=" \theta_{a,b,f}^{(\nu)}(q).
\end{eqnarray}
We write
\begin{eqnarray*}
F_f(1-q)=:\sum_{n=0}^{\infty} \xi_f(n)q^n
\end{eqnarray*}
and define \begin{eqnarray}\label{Gfk}
G_{f}^{(\nu)}(k):=\begin{cases}
\dfrac{2}{\sqrt{M}}\displaystyle\sum_{m\!\!\!\!\!\pmod{M}} f(m)\sin\Bigl(\frac{2\pi mk}{M}\Bigr) &\text{if $\nu=0$,}\\\mbox{}\\
\dfrac{2}{\sqrt{M}}\displaystyle\sum_{m\!\!\!\!\!\pmod{M}}f(m)\cos\Bigl(\frac{2\pi mk}{M}\Bigr) &\text{if $\nu=1$.}
\end{cases}
\end{eqnarray}
Assume there exists a smallest positive integer $k_\nu$ such that $G_f^{(\nu)}(k_\nu)\neq 0$. Next, we define
\begin{eqnarray}\label{Mbound}
M_{f,\nu}:=2\cdot \dfrac{\#\{0\leq m\leq M-1: f(m)\neq 0\}}{|G_f^{(\nu)}(k_\nu)|\sqrt{M}}\max_{0\leq m\leq M-1}|f(m)|
\end{eqnarray}
and let $N_{f, \nu}^{\text{(max)}} \geq 0$ be the smallest integer such that 
\begin{eqnarray}\label{Gbound2}
M_{f,\nu}\left(\zeta(2n+\nu+1)-1\right)<1
\end{eqnarray}
for $n \geq N_{f,\nu}^{(\text{max})}$ where $\zeta(s)$ is the Riemann zeta function. Observe that $N_{f,\nu}^{(\text{max})}$ exists since $\zeta(2n+\nu+1)-1 \to 0$ as $n\to\infty$ and $M_{f,\nu}$ is independent of $n$. Finally, let $B_k(x)$ denote the $k$th Bernoulli polynomial for an integer $k\geq 0$. Our main result is now the following.

\begin{theorem}\label{main}
Assume (\ref{gen-strange}) is true. Then as $n\rightarrow\infty$, we have
\begin{eqnarray} \label{mainasy}
\xi_f(n)\sim  (-1)^\nu\left(\dfrac{M}{2\pi k_\nu}\right)^{2n+\nu+1}\dfrac{G_{f}^{(\nu)}(k_\nu)2^{2n+\nu}n! n^{\nu-\frac{1}{2}}}{b^n\sqrt{\pi M}}e^{\frac{bk_\nu^2\pi^2}{2M^2}}.
\end{eqnarray}
Moreover, there exists an integer $0\leq N_{f,\nu} \leq N_{f,\nu}^{(\text{max})}$ such that if 
\begin{equation}\label{Gbound}
(-1)^{n+1}\sum_{m=1}^M f(m)B_{2n+\nu+1}\left(\dfrac{m}{M}\right)
\end{equation}
has the same sign as $(-1)^{\nu} G_{f}^{(\nu)}(k_{\nu})$ for all $0\leq n < N_{f,\nu}$, then for all non-negative integers $\ell$, $\xi_f(\ell)$ has the same sign as $(-1)^{\nu} G_{f}^{(\nu)}(k_{\nu})$. 
\end{theorem}

The paper is organized as follows. In Section 2, we prove Theorem \ref{main}. In Section 3, we give some applications, including asymptotics and positivity statements for the generalized Fishburn numbers $\xi_t(n)$ which arise from the Kontsevich-Zagier series $\mathscr{F}_t(q)$ associated to the colored Jones polynomial for the family of torus knots $T(3, 2^t)$, $t \geq 2$ \cite{Be}. This extends (\ref{fishasy}) and gives an alternative proof of the positivity of the Fishburn numbers $\xi(n)$ (see Corollary \ref{genFishasymp1} and Remark \ref{ev1}). In Section 4, we comment on asymptotics for other expansions of $F_f(q)$ and then conclude with conjectures concerning the positivity of coefficients for these expansions in three situations: $\mathscr{F}_t(q)$, the Kontsevich-Zagier series associated to the colored Jones polynomial for the family of torus knots $T(2, 2m+1)$, $m \geq 1$ and a ``Habiro-type" $q$-series with origins in Ramanujan's lost notebook \cite{Andrews}.

\section{Proof of Theorem \ref{main}}

\begin{proof}[Proof of Theorem \ref{main}]
We follow the strategy of \cite{z1}. To find the asymptotics of $\xi_f(n)$, we first consider the expansions
\begin{eqnarray}\label{gen-expansions1}
F_f(e^{-t})=\sum_{n=0}^\infty\dfrac{B_{n,f}}{n!}t^n
\end{eqnarray}
and
\begin{eqnarray} \label{gen-expansions2}
e^{\frac{-ta}{b}}F_f(e^{-t})=\sum_{n=0}^\infty\dfrac{C_{n,f}}{n!}\left(\frac{t}{b}\right)^n.
\end{eqnarray}
Let 
\begin{eqnarray*}
\mathcal{P}_{a,b,f}^{(\nu)}(q) :=q^{\frac{a}{b}} \theta_{a,b,f}^{(\nu)}(q)
\end{eqnarray*}
and define the $L$-function
\begin{eqnarray*}
L(s,f) =\sum_{n=1}^\infty\dfrac{f(n)}{n^s}.
\end{eqnarray*}
The Mellin transform of $\mathcal{P}_{a,b,f}^{(\nu)}(e^{-t})$ is
\begin{eqnarray}\label{Mellin}
\int_0^\infty t^{s-1}\mathcal{P}_{a,b,f}^{(\nu)}(e^{-t})dt&=&\sum_{n=0}^{\infty} n^\nu f(n)\int_0^\infty t^{s-1}e^{\frac{-tn^2}{b}}dt\notag\\
&=&b^s\Gamma(s)\sum_{n=0}^{\infty} \dfrac{f(n)}{n^{2s-\nu}} \notag\\
&=& b^s\Gamma(s)L(2s-\nu,f)
\end{eqnarray}
where $\Gamma(s)$ is the usual Gamma function. Applying Mellin inversion to \eqref{Mellin}, we obtain
\begin{eqnarray*}
\mathcal{P}_{a,b,f}^{(\nu)}(e^{-t})=\dfrac{1}{2\pi i}\int_{\operatorname{Re}(s)=c}b^s\Gamma(s)L(2s-\nu,f)\dfrac{ds}{t^s}
\end{eqnarray*}
where $\operatorname{Re}(s)=c$ is the abscissa of absolute convergence of $\frac{b^s\Gamma(s)L(2s-\nu,f)}{t^s}$. It is well-known that $L(s,f)$ can be analytically continued to the whole complex plane except for a possible simple pole at $s=1$ with residue 
\begin{eqnarray*}
R_{f,M}:=\dfrac{1}{M}\sum_{m\!\!\!\!\!\pmod{M}}f(m).
\end{eqnarray*}
By a standard complex analytic computation, we have
\begin{eqnarray*}
\mathcal{P}_{a,b,f}^{(\nu)}(e^{-t})\sim
\displaystyle\sum_{n=0}^\infty \dfrac{(-1)^nL(-2n-\nu,f)}{b^nn!}t^n+\dfrac{b^{\frac{\nu+1}{2}}\Gamma\left(\frac{\nu+1}{2}\right)R_{f,M}}{t^{\frac{\nu+1}{2}}}
\end{eqnarray*}
as  $t \rightarrow 0^+$. By  \eqref{gen-strange} and \eqref{gen-expansions2}, we compare coefficients to obtain
\begin{eqnarray}\label{Cnf}
C_{n,f}=(-1)^nL(-2n-\nu,f).
\end{eqnarray}
Also, it is clear from \eqref{gen-strange} and \eqref{gen-expansions2} that $R_{f,M}=0$. Via \cite[Eqn. (24)]{schnee} or \cite[Chapter 12]{apostol}, we have 
\begin{eqnarray}\label{Lfunc}
L(-2n-\nu,f)=(-1)^{n+\nu}\left(\dfrac{M}{2\pi}\right)^{2n+\nu+1}\dfrac{(2n+\nu)!}{\sqrt{M}}\sum_{k=1}^\infty\dfrac{G_{f}^{(\nu)}(k)}{k^{2n+\nu+1}}.
\end{eqnarray}
Let $k_\nu\geq 1$ be the smallest integer such that $G_f^{(\nu)}(k_\nu)\neq 0$. Then
\begin{eqnarray}\label{gen-L}
L(-2n-\nu,f)=(-1)^{n+\nu}\left(\dfrac{M}{2\pi}\right)^{2n+\nu+1}\dfrac{G_{f}^{(\nu)}(k_\nu)(2n+\nu)!}{k_\nu^{2n+\nu+1}\sqrt{M}}\left(1+O\left(\left(\frac{k_\nu}{k_\nu+1}\right)^{2n-\varepsilon}\right)\right)
\end{eqnarray}
for any $\varepsilon>0$\footnote{We can take $\varepsilon=0$ when $\nu=1$.}. Using \eqref{Cnf} and \eqref{gen-L}, it follows 
\begin{eqnarray}\label{Cnf-asymp}
C_{n,f}=(-1)^\nu\left(\dfrac{M}{2\pi}\right)^{2n+\nu+1}\dfrac{G_{f}^{(\nu)}(k_\nu)(2n+\nu)!}{k_\nu^{2n+\nu+1}\sqrt{M}}\left(1+O\left(\left(\frac{k_\nu}{k_\nu+1}\right)^{2n-\varepsilon}\right)\right).
\end{eqnarray}
From \eqref{gen-expansions1}, we have
\begin{eqnarray}\label{BCnf-relation}
B_{n,f}=\dfrac{1}{b^n}\sum_{k=0}^n\binom{n}{k}a^{n-k}C_{k,f}=\dfrac{1}{b^n}\left(C_{n,f}+ naC_{n-1,f}+\dfrac{n(n-1)a^2}{2}C_{n-2,f}+\cdots\right).
\end{eqnarray}
From \eqref{Cnf-asymp} and \eqref{BCnf-relation}, we deduce
\begin{eqnarray}\label{Bnf-asymp1}
&&B_{n,f}=\dfrac{C_{n,f}}{b^n}\left(1+\dfrac{na \left(\frac{2\pi k_\nu}{M}\right)^2}{(2n+\nu-1)(2n+\nu)}\right.\notag\\&&\hspace{3.5cm}\left.+\;\dfrac{n(n-1)a^2 \left(\frac{2\pi k_\nu}{M}\right)^4}{2(2n+\nu-3)(2n+\nu-2)(2n+\nu-1)(2n+\nu)}+\cdots\right).
\end{eqnarray}
An application of Stirling's formula 
\begin{equation}\label{Stirfor1}
n!=\sqrt{2\pi n} \left(\dfrac{n}{e}\right)^n \left(1 +\dfrac{1}{12n}+\frac{1}{288n^{2}} + \cdots \right)
\end{equation}
implies
\begin{align}\label{Stir1}
\dfrac{(2n+\nu)!}{n!^2} = \dfrac{2^{2n+\nu} n^{\nu-\frac{1}{2}}}{\sqrt{\pi}}\left(1+\dfrac{(-1)^{\nu+1}(2\nu+1)}{8n}+ \dfrac{(-1)^\nu(6\nu+1)}{128n^2}+\cdots\right).  
\end{align}
Combining \eqref{Cnf-asymp}, \eqref{Bnf-asymp1} and \eqref{Stir1} yields
\begin{align*}
B_{n,f}&=(-1)^\nu\left(\dfrac{M}{2\pi k_\nu}\right)^{2n+\nu+1}\dfrac{G_{f}^{(\nu)}(k_\nu)2^{2n+\nu} n!^2 n^{\nu-\frac{1}{2}}}{b^n\sqrt{\pi M}}\left(1+O\left(\left(\frac{k_\nu}{k_\nu+1}\right)^{2n-\varepsilon}\right)\right)\notag\\
&\times\left(1+\dfrac{(-1)^{\nu+1}(2\nu+1)}{8n}+ \dfrac{(-1)^\nu(6\nu+1)}{128n^2}+\cdots\right)\left(1+\dfrac{na \left(\frac{2\pi k_\nu}{M}\right)^2}{(2n+\nu-1)(2n+\nu)}\right. \notag \\
&\hspace{3.5cm}\left.+\;\dfrac{n(n-1) a^2 \left(\frac{2\pi k_\nu}{M}\right)^4}{2(2n+\nu-3)(2n+\nu-2)(2n+\nu-1)(2n+\nu)}+\cdots\right) \nonumber \\
\end{align*}
and so
\begin{eqnarray} \label{key}
B_{n,f} \sim (-1)^{\nu}\left(\dfrac{M}{2\pi k_\nu}\right)^{2n+\nu+1}\dfrac{G_{f}^{(\nu)}(k_\nu)2^{2n+\nu} n!^2 n^{\nu-\frac{1}{2}}}{b^n\sqrt{\pi M}}\left(1+\dfrac{\alpha_{1,f,\nu}}{n}+\dfrac{\alpha_{2,f,\nu}}{n^2}+\cdots\right)
\end{eqnarray}
where $\alpha_{1,f,\nu}:=\frac{a\left(\frac{2\pi k_\nu}{M}\right)^2}{4}+\frac{(-1)^{\nu+1}(2\nu+1)}{8}$ and all the remaining constants $\alpha_{i,f,\nu}$ are effectively computable. Now, we recall \begin{eqnarray}\label{gfStir1}
\dfrac{t^m}{m!}=\sum_{n=m}^\infty S_{n,m}\dfrac{(1-e^{-t})^n}{n!}
\end{eqnarray}
where $S_{n,m}$ denotes the Stirling numbers of the first kind. From \eqref{gen-expansions1} and \eqref{gfStir1}, we interchange sums
\begin{eqnarray*}\label{FBS}
F_f(e^{-t})=\sum_{m=0}^\infty B_{m,f}\sum_{n=m}^\infty S_{n,m}\dfrac{(1-e^{-t})^n}{n!}= B_{0,f}  + \sum_{n=1}^\infty \dfrac{(1-e^{-t})^n}{n!}\sum_{m=1}^{n}S_{n,m}B_{m,f}
\end{eqnarray*}
and thus
\begin{eqnarray}\label{xi-f}
\xi_f(n)=\dfrac{1}{n!}\sum_{m=0}^{n-1}S_{n,n-m}B_{n-m,f}.
\end{eqnarray}
Next, from $S_{n,n}=1$ and the recursion $S_{n+1,m}=S_{n,m-1}+nS_{n,m}$ we have
\begin{eqnarray}\label{Stirasymp}
S_{n,n-m}=\dfrac{n^{2m}}{2^mm!}\left(1-\dfrac{\beta_1(m)}{n}+\dfrac{\beta_2(m)}{n^2}+\cdots\right)
\end{eqnarray}
with computable coefficients $\beta_1(m)=\frac{2m^2+m}{3},\;\beta_2(m),\cdots$. From \eqref{key}, it follows
\begin{eqnarray}
\dfrac{B_{n-m,f}}{B_{n,f}} = \left(\dfrac{bk_\nu^2\pi^2}{M^2}\right)^{m}\left(\dfrac{(n-m)!}{n!}\right)^2\left(1+O\left( \dfrac{1}{n}  \right)\right)
\end{eqnarray}
and by (\ref{Stirfor1})
\begin{eqnarray}\label{est21}
\left(\dfrac{(n-m)!}{n!}\right)^2=\frac{1}{n^{2m}} \left(1+\dfrac{m(m-1)}{n}+\dfrac{3m^4-4m^3+m}{6n^2}+\cdots\right).
\end{eqnarray}
Finally, we combine \eqref{xi-f}--\eqref{est21} to obtain 
\begin{align*}
\xi_f(n) &= \dfrac{B_{n,f}}{n!}\sum_{m=0}^{n-1}\dfrac{1}{m!}\left(\dfrac{bk_\nu^2\pi^2}{2M^2}\right)^m\left(1+O\left( \dfrac{1}{n} \right)\right)\notag\\&\qquad \qquad \qquad \times\left(1-\dfrac{\beta_1(m)}{n}+\dfrac{\beta_2(m)}{n^2}+\cdots\right)\left(1+\dfrac{m(m-1)}{n}+\dfrac{3m^4-4m^3+m}{6n^2}+\cdots\right)\notag\\
&= \dfrac{B_{n,f}}{n!}\left(\sum_{m=0}^\infty\dfrac{1}{m!}\left(\dfrac{bk_\nu^2\pi^2}{2M^2}\right)^m+O\left(\dfrac{1}{n}\right)\right) \notag \\
&= \dfrac{B_{n,f}}{n!}e^{\frac{bk_\nu^2\pi^2}{2M^2}}\left(1+O\left(\dfrac{1}{n}\right)\right).
\end{align*}
The result (\ref{mainasy}) now follows from (\ref{key}). Now for fixed $f$, we have from \eqref{Cnf} and \eqref{Lfunc} 
\begin{equation}\label{CtoG}
C_{n,f} =(-1)^\nu G_f^{(\nu)}(k_\nu)\left(\dfrac{M}{2\pi}\right)^{2n+\nu+1}\dfrac{(2n+\nu)!}{k_\nu^{2n+\nu+1}\sqrt{M}}\left(1+ \frac{1}{G_f^{(\nu)}(k_\nu)} \sum_{k=k_{\nu} + 1}^\infty\dfrac{G_f^{(\nu)}(k)}{k^{2n+\nu+1}}\right).
\end{equation}
Next, \eqref{Gfk} implies 
\begin{eqnarray*}
\left | \dfrac{G_f^{(\nu)}(k)}{G_f^{(\nu)}(k_\nu)} \right | \leq M_{f,\nu}
\end{eqnarray*}
and this yields
\begin{eqnarray}\label{Gbound1}
\left| \frac{1}{G_f^{(\nu)}(k_\nu)} \sum_{k=k_{\nu} + 1}^\infty\dfrac{G_f^{(\nu)}(k)}{k^{2n+\nu+1}}\right|\leq M_{f,\nu}\left(\zeta(2n+\nu+1)-1\right)<1
\end{eqnarray}
for $n \geq N_{f, \nu}^{\text{(max)}}$ where $N_{f, \nu}^{\text{(max)}}$ is as in \eqref{Gbound2}. Clearly, we can choose $0\leq N_{f, \nu} \leq N_{f, \nu}^{\text{(max)}}$ satisfying \eqref{Gbound1} for $n\geq N_{f,\nu}$. For such an $N_{f, \nu}$, it now follows from (\ref{CtoG}) that $C_{n,f}$ and $(-1)^{\nu} G_{f}^{(\nu)}(k_{\nu})$ have the same sign for all $n \geq N_{f,\nu}$. Next, we note using \eqref{Cnf} and \cite[Lemma 3.2]{akl} that
\begin{eqnarray}\label{CnfBer}
C_{n,f}=(-1)^{n+1}\dfrac{M^{2n+\nu}}{2n+\nu+1}\sum_{m=1}^M f(m)B_{2n+\nu+1}\left(\dfrac{m}{M}\right)
\end{eqnarray}
where for $k\geq 0$, $B_k(x)$ denotes the $k$th Bernoulli polynomial. Hence \eqref{CnfBer} implies that if $C_{n,f}$, or equivalently \eqref{Gbound}, has the same sign as $(-1)^{\nu} G_{f}^{(\nu)}(k_{\nu})$ for $0 \leq n < N_{f, \nu}$, then \eqref{BCnf-relation} and \eqref{xi-f} imply that $\xi_f(\ell)$ has the same sign as $(-1)^{\nu} G_{f}^{(\nu)}(k_{\nu})$ for all $\ell \geq 0$.
\end{proof}

\section{Examples}

In this section, we illustrate Theorem \ref{main} with three examples.

\subsection{Kontsevich-Zagier series for torus knots $T(3,2^t)$}
For $t \geq 2$, consider the Kontsevich-Zagier series associated to the family of torus knots $T(3,2^t)$
\begin{align} \label{kztorus} 
\mathscr{F}_t(q) & = (-1)^{h''(t)} q^{-h'(t)} \sum_{n=0}^{\infty} (q)_n \sum_{3 \sum_{\ell = 1}^{m(t) - 1} j_{\ell} \ell \,\equiv \,1\;(\tiny{\mbox{mod}} \; m(t))} (-1)^{\sum_{\ell = 1}^{m(t)-1} j_{\ell}}q^{\frac{-a(t) + \sum_{\ell=1}^{m(t) - 1} j_{\ell} \ell}{m(t)} + \sum_{\ell=1}^{m(t)-1} \binom{j_{\ell}}{2}} \nonumber \\
& \times \sum_{k=0}^{m(t)-1} \prod_{\ell=1}^{m(t) - 1} \begin{bmatrix} n + I(\ell \leq k) \\ j_{\ell} \end{bmatrix}
\end{align}
where 
\begin{eqnarray*}
h''(t)=\left\{\begin{array}{cc}
\frac{2^t-1}{3},&\mbox{if}\;t\;\mbox{is even},\\
\frac{2^t-2}{3},&\mbox{if}\;t\;\mbox{is odd},
\end{array}\right.\hspace{0.5cm}
h'(t)=\left\{\begin{array}{cc}
\frac{2^t-4}{3},&\mbox{if}\;t\;\mbox{is even},\\
\frac{2^t-5}{3},&\mbox{if}\;t\;\mbox{is odd},
\end{array}\right.\hspace{0.5cm}
a(t)=\left\{\begin{array}{cc}
\frac{2^{t-1}+1}{3},&\mbox{if}\;t\;\mbox{is even},\\
\frac{2^{t}+1}{3},&\mbox{if}\;t\;\mbox{is odd},
\end{array}\right.
\end{eqnarray*}
$m(t)=2^{t-1}$, $I(*)$ is the characteristic function and 
\begin{eqnarray*}
\begin{bmatrix}
n\\k
\end{bmatrix}
=\dfrac{(q)_n}{(q)_{n-k}(q)_k}
\end{eqnarray*}
is the $q$-binomial coefficient.\footnote{For $t=1$, one may define the sum over the $j_{\ell}$ to be 1 in (\ref{kztorus}) to recover (\ref{f}).} The expression $\mathscr{F}_t(q)$ matches the $N$th colored Jones polynomial for $T(3, 2^t)$ at a root of unity $q=e^{\frac{2\pi i}{N}}$, converges in a similar manner as $F(q)$ and is an element of $\mathcal{H}$ (see \cite{Be} for further details). The \textit{generalized Fishburn numbers} $\xi_t(n)$ are defined by
\begin{eqnarray*}
\mathscr{F}_t(1-q)= \sum_{n=0}^\infty\xi_t(n)q^n.
\end{eqnarray*}
An application of Theorem \ref{main} is the following. Note that (\ref{fishasy}) follows after taking $t=1$ and simplifying (for brevity, we only state the leading term). 
\begin{corollary}\label{genFishasymp1}
Let $t \geq 1$. As $n\rightarrow\infty$, we have 
\begin{equation} \label{genFish1}
\xi_t(n)\sim \dfrac{\sin(\frac{\pi}{2^t})}{2^t\sqrt{3\pi}}\left(\dfrac{3\cdot 2^t}{\pi}\right)^{2n+2}\dfrac{2^{2n+1}n!\sqrt{n}}{(3\cdot 2^{t+2})^n}e^{\frac{\pi^2}{3\cdot 2^{t+1}}}.
\end{equation}
Moreover, let $N_t \geq 0$ be the smallest integer such that 
\begin{eqnarray}\label{Gtbound}
\zeta(2n+2)<\sin\left(\dfrac{\pi}{2^t}\right)+1 
\end{eqnarray}
for $n\geq N_t$. If
\begin{equation}\label{Gbound3}
(-1)^{n}\left[B_{2n+2}\left(\frac{2^{t+1}-3}{3\cdot 2^{t+1}}\right)-B_{2n+2}\left(\frac{2^{t+1}+3}{3\cdot 2^{t+1}}\right)\right] \geq 0
\end{equation}
for all $0\leq n<N_t$, then $\xi_t(\ell) > 0$ for all $\ell \geq 0$ and $t \geq 1$.
\end{corollary}

\begin{proof}
The Kontsevich-Zagier series $\mathscr{F}_t(q)$ satisfies the strange identity \cite[Proposition 2.4]{Be}\footnote{Taking $t=1$ in (\ref{KZstrange}) and (\ref{chit}) recovers (\ref{strid}).}
\begin{eqnarray} \label{KZstrange}
\mathscr{F}_t(q)``="\theta^{(1)}_{(2^{t+1} - 3)^2, 3\cdot 2^{t+2}, \chi_{t}}(q)
\end{eqnarray}
where
\begin{equation} \label{chit}
\chi_t(n) :=
\begin{cases}
-\frac{1}{2} &\text{if $n \equiv 2^{t+1}-3$, $3+2^{t+2} \; (\text{mod}\; 3\cdot2^{t+1}),$} \\
\frac{1}{2} &\text{if $n \equiv 2^{t+1} + 3$, $2^{t+2} - 3 \; (\text{mod}\; 3\cdot2^{t+1}),$} \\
0 &\text{otherwise.}
\end{cases}
\end{equation} 
Note that $\chi_t$ is an even function with period $M=3\cdot 2^{t+1}$. Next, we claim that $k_1=1$. To see this, observe that \eqref{Gfk} and \eqref{chit} yield
\begin{eqnarray}\label{Gkt}
G_{\chi_t}^{(1)}(1)&=&-\dfrac{1}{\sqrt{3\cdot 2^{t+1}}}\left\{\cos\left(\frac{2\pi(2^{t+1}-3)}{3\cdot 2^{t+1}}\right)+\cos\left(\frac{2\pi(3+2^{t+2})}{3\cdot 2^{t+1}}\right)\right.\notag\\&&\hspace{1cm}\left.-\cos\left(\frac{2\pi(2^{t+1}+3)}{3\cdot 2^{t+1}}\right)-\cos\left(\frac{2\pi(2^{t+2}-3)}{3\cdot 2^{t+1}}\right)\right\} \notag \\
&=& -\dfrac{1}{\sqrt{2^{t-1}}}\sin\left(\dfrac{\pi}{2^t}\right)
\end{eqnarray}
which is non-zero for any $t\geq 1$. By Theorem \ref{main}, (\ref{KZstrange}) and (\ref{Gkt}), (\ref{genFish1}) follows. To deduce the positivity statement for $\xi_{t}(n)$, we first note using \eqref{Mbound} and \eqref{Gkt} that
\begin{eqnarray*}
M_{\chi_t,1}=\dfrac{8}{\sqrt{3\cdot 2^{t+1}} \,\, \Big \lvert G_{\chi_t}^{(1)}(k_\nu) \Big \lvert}=\dfrac{4}{\sqrt{3}\sin\left(\frac{\pi}{2^t}\right)}.
\end{eqnarray*}
Thus, \eqref{Gbound2} implies that $N_t^{(\text{max})}:=N_{\chi_t,1}^{(\text{max})}\geq 0$ is the smallest integer satisfying
\begin{eqnarray*}
\zeta(2n+2)<\dfrac{\sqrt{3}}{4}\sin\left(\dfrac{\pi}{2^t}\right)+1
\end{eqnarray*}
for $n \geq N_t^{(\text{max})}$. In fact,
\begin{eqnarray*}
\Big \lvert G_{\chi_t}^{(1)}(k) \Big \rvert \leq \dfrac{1}{\sqrt{2^{t-1}}}
\end{eqnarray*}
and thus we obtain using \eqref{Gkt}
\begin{eqnarray*}
\left|\dfrac{1}{G_{\chi_t}^{(1)}(1)}\sum_{k=2}^\infty\dfrac{G_{\chi_t}^{(1)}(k)}{k^{2n+2}}\right|\leq \dfrac{\zeta(2n+2)-1}{\sin\left(\frac{\pi}{2^t}\right)}<1
\end{eqnarray*}
for $n\geq N_t$ with $0\leq N_t\leq N_t^{\text{(max)}}$. Using Theorem \ref{main} and the fact that 
\begin{eqnarray} \label{bk}
B_k(x)=(-1)^kB_k(1-x)
\end{eqnarray}
for $k\geq 0$, it now follows that for all non-negative integers $\ell$, $\xi_t(\ell)> 0$ if \eqref{Gbound3} is non-negative for $0\leq n \,<\,  N_t$. 
\end{proof}

\begin{remark} \label{ev1}
Using Corollary \ref{genFishasymp1}, we verify \eqref{Gbound3} for $0\leq n < N_t$ to confirm that $\xi_t(\ell)> 0$ for all $\ell \geq 0$ and $1 \leq t \leq 500$. In particular, this shows that $\xi(\ell)>0$ for all $\ell \geq 0$. In Table \ref{Nt}, we list values of $N_t$ for $1 \leq t \leq 10$.

\begin{table}[h]
{\tabulinesep=1mm
\begin{tabu}{|c|l|l|l|l|l|l|l|l|l|l|} \hline
$t$ & 1 & 2 & 3 & 4 & 5 & 6 & 7 & 8 & 9 & 10 \\ \hline
$N_t$ & 0 & 0 & 1 & 1 & 1 & 2 & 2 & 3 & 3 & 4 \\ \hline
\end{tabu}} 
\captionsetup{justification=centering}
\caption{\label{Nt} List of values of $N_t$ for $1 \leq t \leq 10$}
\end{table}
\end{remark}
\subsection{Kontsevich-Zagier series for torus knots $T(2,2m+1)$}
Let $m\in\mathbb{N}$. For $0\leq \ell\leq m-1$, define the Kontsevich-Zagier series for the torus knot $T(2,2m+1)$ as follows:
\begin{eqnarray*}
X_m^{(\ell)}(q):=\sum_{k_1,k_2,\dotsc,k_m=0}^\infty (q)_{k_m} q^{k_1^2+\cdots+k_{m-1}^2+k_{\ell+1}+\cdots+k_{m-1}}\prod_{i=1}^{m-1}\begin{bmatrix} k_{i+1} + \delta_{i,\ell} \\ k_{i} \end{bmatrix}
\end{eqnarray*}
where $\delta_{i,\ell}$ is the characteristic function. The expression $X_m^{(\ell)}(q)$ matches the $N$th colored Jones polynomial for $T(2, 2m+1)$ when $\ell=0$ and $q=e^{\frac{2\pi i}{N}}$ and is an element of $\mathcal{H}$. Write 

\begin{eqnarray*}
X_m^{(\ell)}(1-q) =: \sum_{n=0}^{\infty} \xi_{\ell, m}(n) q^n.
\end{eqnarray*}
Another application of Theorem \ref{main} is the following. Observe that (\ref{fishasy}) also follows by choosing $m=1$ and $\ell = 0$ and simplifying as $X_{1}^{(0)}(q) = F(q)$.

\begin{corollary} \label{genFishasymp2}
Let $m\in\mathbb{N}$ and $0\leq \ell\leq m-1$. As $n\rightarrow\infty$, we have 
\begin{equation} \label{genFish2}
\xi_{\ell, m}(n)\sim \sin\left(\frac{\pi(\ell+1)}{2m+1}\right)\left(\dfrac{2m+1}{\pi^2}\right)^{n+1}\dfrac{2^{n+3}n!\sqrt{n}}{\sqrt{\pi}}e^{\frac{\pi^2}{8m+4}}.
\end{equation}
Moreover, let $N_{m,\ell} \ge 0$ be the smallest integer such that 
\begin{eqnarray}\label{GXbound}
\zeta(2n+2)<\sin\left(\dfrac{\pi(\ell+1)}{2m+1}\right)+1 
\end{eqnarray}
for $n\geq N_{m,\ell}$. If
\begin{equation}\label{Gbound4}
(-1)^{n}\left[B_{2n+2}\left(\frac{2m-2\ell-1}{8m+4}\right)-B_{2n+2}\left(\frac{2m+2\ell+3}{8m+4}\right)\right] \geq 0
\end{equation}
for all $0\leq n<N_{m,\ell}$, then $\xi_{\ell,m}(k) > 0$ for all $k \geq 0$ and $0\leq \ell\leq m-1$.
\end{corollary}

\begin{proof}
Hikami \cite[Eqn. (15)]{hikami2} established the strange identity
\begin{eqnarray}\label{strangehikami}
X_m^{(\ell)}(q)``=" \theta^{(1)}_{(2m-2\ell-1)^2, 8(2m+1), \chi_{m}^{(\ell)}}(q)
\end{eqnarray}
where 
\begin{equation} \label{genchi}
\chi_{m}^{(\ell)}(n):=
\begin{cases}
-\frac{1}{2} &\text{if $n \equiv 2m-2\ell-1$, $6m+2\ell+5$ $\pmod{8m+4}$,} \\
\frac{1}{2} &\text{if $n \equiv 2m+2\ell+3$, $6m-2\ell+1$ $\pmod{8m+4}$,} \\
0 &\text{otherwise.}
\end{cases}
\end{equation}
Note that $\chi_{m}^{(\ell)}(n)$ is an even function with period $M = 8m+4$. Next, we claim that $k_1 = 1$. To see this, observe that \eqref{Gfk} and \eqref{genchi} yield
\begin{eqnarray}\label{GXt}
G_{\chi_{m}^{(\ell)}}^{(1)}(1)&=&-\dfrac{1}{\sqrt{8m+4}}\left\{\cos\left(\frac{2\pi(2m-2\ell-1)}{8m+4}\right)+\cos\left(\frac{2\pi(6m+2\ell+5)}{8m+4}\right)\right.\notag\\
&&\hspace{1cm}\left.-\cos\left(\frac{2\pi(2m+2\ell+3)}{8m+4}\right)-\cos\left(\frac{2\pi(6m-2\ell+1)}{8m+4}\right)\right\} \notag \\
&=& -\dfrac{2}{\sqrt{2m+1}}\sin\left(\frac{\pi(\ell+1)}{2m+1}\right)
\end{eqnarray}
which is non-zero for any $m\in\mathbb{N}$ and $0\leq \ell\leq m-1$. By Theorem \ref{main}, (\ref{strangehikami}) and (\ref{GXt}), (\ref{genFish2}) follows. To deduce the positivity statement for $\xi_{\ell, m}(n)$, we first note using \eqref{Mbound} and \eqref{GXt} that
\begin{eqnarray*}
M_{\chi_{m}^{(\ell)},1}=\dfrac{8}{\sqrt{8m+4} \,\, \Big \lvert G_{\chi_{m}^{(\ell)}}^{(1)}(k_\nu) \Big \rvert } =\dfrac{2}{\sin\left(\frac{\pi(\ell+1)}{2m+1}\right)}.
\end{eqnarray*}
Thus, \eqref{Gbound2} implies that $N_{m,\ell}^{(\text{max})}:=N_{\chi_{8m+4}^{(\ell)},1}^{(\text{max})}\geq 0$ is the smallest integer satisfying 
\begin{eqnarray}
\zeta(2n+2)<\dfrac{1}{2}\sin\left(\dfrac{\pi(\ell+1)}{2m+1}\right)+1
\end{eqnarray}
for $n \geq N_{m,\ell}^{(\text{max})}$. In fact,
\begin{eqnarray*}
\Big \lvert  G_{\chi_{m}^{(\ell)}}^{(1)}(k) \Big \rvert  \leq \dfrac{2}{\sqrt{2m+1}}
\end{eqnarray*}
and thus we obtain using \eqref{GXt}
\begin{eqnarray*}
\left|\dfrac{1}{G_{\chi_{m}^{(\ell)}}^{(1)}(1)}\sum_{k=2}^\infty\dfrac{G_{\chi_{m}^{(\ell)}}^{(1)}(k)}{k^{2n+2}}\right|\leq \dfrac{\zeta(2n+2)-1}{\sin\left(\frac{\pi(\ell+1)}{2m+1}\right)}<1
\end{eqnarray*}
for $n\geq N_{m,\ell}$ with $0\leq N_{m,\ell}\leq N_{m,\ell}^{\text{(max)}}$. Using Theorem \ref{main} and (\ref{bk}), it now follows that for all non-negative integers $k$, $\xi_{\ell,m}(k)> 0$ if \eqref{Gbound4} is non-negative for $0\leq n < N_{m,\ell}$. 
\end{proof}

\begin{remark} \label{ev2}
Using Corollary \ref{genFishasymp2}, we verify \eqref{Gbound4} for $0\leq n < N_{m,\ell}$ to confirm that $\xi_{\ell,m}(k)> 0$ for all $k \geq 0$, $1 \leq m \leq 500$ and $0\leq \ell \leq m-1$. In Table \ref{Nml}, we list values of $N_{m,\ell}$ for $1 \leq m \leq 5$ and $0\leq \ell \leq m-1$.

\begin{table}[h]
\begin{tabular}{|l|l|l|l|l|l|l|l|l|l|l|l|l|l|l|l|}
 \hline
 \multicolumn{1}{|c|}{$m$} &
 \multicolumn{1}{|c|}{$1$} &
 \multicolumn{2}{|c|}{$2$} &
 \multicolumn{3}{|c|}{$3$} &
 \multicolumn{4}{|c|}{$4$} &
 \multicolumn{5}{|c|}{$5$} \\
 \hline
 $\ell$ & 0 & 0 & 1 & 0 & 1 & 2 & 0 & 1 & 2 & 3 & 0 & 1 & 2 & 3 & 4\\
 \hline
 $N_{m,\ell}$ & 0 & 1 & 0 & 1 & 0 & 0 & 1 & 1 & 0 & 0 & 1 & 1 & 0 & 0 & 0\\
 \hline
\end{tabular}
\captionsetup{justification=centering}
\caption{\label{Nml} List of values of $N_{m,\ell}$ for $1 \leq m \leq 5$ and $0\leq \ell \leq m-1$}
\end{table}
\end{remark}
\FloatBarrier

\begin{remark} \label{checklm}
In fact, we can determine an infinite number of $m$ and $0\leq \ell \leq m-1$ such that $\xi_{\ell,m}(k)$ is positive for every $k \geq 0$. Let us put $\ell=cm+d$. Then
\begin{enumerate}
\item It turns out that in order for $\ell\in\mathbb{Z}$, $c$ and $d$ must be rational numbers in their reduced forms such that $c=\frac{p_1}{q_1}$ and $d=\frac{p_2}{q_1}$ so that
\begin{eqnarray}\label{cond1}
p_1m\equiv -p_2\!\!\!\!\pmod{q_1}.
\end{eqnarray}
    \item Let $m_0$ be the smallest non-negative integer satisfying the congruence in \eqref{cond1}. Then with the choices of $c$ and $d$ as in (1) and using \eqref{GXbound} with $N_{m,\ell}=0$, we have
    \begin{eqnarray}\label{cond2}
    \max\left(0,\frac{2m_0-3}{4}\right)\leq cm_0+d \leq m_0-1, \quad \text{and}\quad \frac{1}{2}\leq c\leq 1.  
    \end{eqnarray}
\end{enumerate}
Equations \eqref{cond1} and \eqref{cond2} can now be used to determine an infinite family of $m$ and $0\leq \ell \leq m-1$ such that $\xi_{\ell,m}(k)>0$ for all $k\geq 0$ and $m\geq 1$. For example, let us choose $c=1\;(p_1=1,\;q_1=1)$. Then \eqref{cond1} and \eqref{cond2} force $m_0=1$ and $d=-1$. Thus, $\xi_{m-1,m}(k)>0$ for all $k$. Similarly, if we choose $c=\frac{1}{2}\;(p_1=1,\;q_1=2)$, then \eqref{cond1} implies that $m\equiv 1\pmod{2}$. This combined with \eqref{cond2} force $m_0=1$ and $d=-\frac{1}{2}$ so that we have $\xi_{\frac{m-1}{2},m}(k)>0$ for all $k\geq 0$ and integers $m\equiv 1\!\!\pmod{2}$. 
\end{remark}
 
\subsection{An example with $\nu=0$} For $k\geq 1$, let $\mathscr{G}_k(q)$ denote the $q$-series
\begin{eqnarray}\label{case0}
\mathscr{G}_k(q):=\sum_{n_k\geq n_{k-1}\geq\cdots\geq n_{1}\geq 0}q^{n_k+2n_{k-1}^2+n_{k-1}+\cdots+2n_1^2+2n_1}(q;q^2)_{n_k}\begin{bmatrix}
n_k\\n_{k-1}
\end{bmatrix}_{q^2}\cdots\begin{bmatrix}
n_2\\n_{1}
\end{bmatrix}_{q^2}
\end{eqnarray}
and write 
\begin{equation*}
\mathscr{G}_k(1-q)=:\sum_{n=0}^{\infty} \xi_{\mathscr{G}_k}(n)q^n.    
\end{equation*}
The $k=1$ case of (\ref{case0}) is of substantial historical and modern importance as it appears in Ramanujan's lost notebook (e.g., see \cite[Section 5]{Andrews}, \cite[Entry 9.5.2]{AB1} or \cite[page 419]{BAJY}).

\begin{corollary} \label{genG}
As $n\rightarrow\infty$, we have
\begin{eqnarray}\label{v0asymp}
\xi_{\mathscr{G}_k}(n)\sim \dfrac{\cos\left(\frac{\pi}{2(2k+1)}\right)2^{2n+2}n!}{\pi^{\frac{3}{2}}\sqrt{n}}\left(\dfrac{2k+1}{\pi^2}\right)^{n}e^{\frac{\pi^2}{8(2k+1)}}.
\end{eqnarray}
Moreover, $\xi_{\mathscr{G}_k}(n)>0$ for all $n\geq 0$ and $k \geq 1$.
\end{corollary}

\begin{proof}
It was shown in \cite[Example 5.2]{akl} that 
\begin{eqnarray}\label{act-strange}
\mathscr{G}_k(q)=\sum_{n=0}^{\infty} \chi_{k}(n)q^{\frac{n^2-k^2}{2k+1}}
\end{eqnarray}
where
\begin{equation}\label{chi4k2}
\chi_{k}(n):=\begin{cases}
1 &n\equiv k,\;k+1\!\!\pmod{4k+2},\\
-1 &n\equiv -k,\;-k-1\!\!\pmod{4k+2},\\
0 &\text{otherwise}.
\end{cases}    
\end{equation}
Observe that \eqref{act-strange} is not a strange identity but an actual identity valid for $|q|<1$ and every odd order root of unity $q$ (see \cite[Example 3.2]{akl}). Although $\mathscr{G}_k(q) \not\in \mathcal{H}$, we note that when $q=e^{-t}$ with $t\rightarrow 0^+$, the expression $(q;q^2)_{n_k}$ in the right-hand side of \eqref{case0} will have an asymptotic expansion starting with $t^{n_k}$. Hence, the expansions \eqref{gen-expansions1} and \eqref{gen-expansions2} for this $q$-series as $t\rightarrow 0^+$ are still valid and we can apply Theorem \ref{main}. First, we have that $\chi_{k}(n)$ is an odd function with period $M = 4k+2$. Next, we claim that $k_0=1$. To see this, observe that for any $\ell\geq 1$, \eqref{Gfk} and \eqref{chi4k2} yield
\begin{eqnarray}\label{Gl}
G_{\chi_{k}}^{(0)}(\ell)&=&\dfrac{4}{\sqrt{4k+2}}\left[\sin\left(\dfrac{\pi k\ell}{2k+1}\right)+\sin\left(\dfrac{\pi (k+1)\ell}{2k+1}\right)\right]\notag\\
&=&\dfrac{8}{\sqrt{4k+2}}\sin\left(\dfrac{\pi\ell}{2}\right)\cos\left(\dfrac{\pi\ell}{2(2k+1)}\right)
\end{eqnarray}
and thus
\begin{eqnarray*}
G_{\chi_{k}}^{(0)}(1)=\frac{8}{\sqrt{4k+2}}\cos\left(\frac{\pi}{2(2k+1)}\right)
\end{eqnarray*} 
which is non-zero for all $k \geq 1$. By Theorem \ref{main}, \eqref{act-strange} and \eqref{Gl}, \eqref{v0asymp} follows. Next, using \eqref{Mbound} we get
\begin{eqnarray*}
M_{\chi_{k},0}=\dfrac{1}{\cos\left(\dfrac{\pi}{2(2k+1)}\right)}.
\end{eqnarray*}
As $\cos\left(\frac{\pi}{2(2k+1)}\right)$ is an increasing function for $k\geq 1$ and
\begin{eqnarray*}
M_{\chi_{1},0}\cdot(\zeta(3)-1)=0.233\cdots<1,
\end{eqnarray*}
we can choose $N_{\chi_{k},0}=1$ for all $k\geq 1$. To deduce the positivity statement for $\xi_{\mathscr{G}_k}(n)$, we need only show that
\begin{eqnarray}\label{nonneg}
\sum_{m=1}^{4k+2}\chi_{k}(m)B_1\left(\frac{m}{4k+2}\right) \leq 0
\end{eqnarray}
for all  $k\geq 1$. To prove \eqref{nonneg}, we first note that $B_1(x)=x-\frac{1}{2}$. Hence, \eqref{chi4k2} implies
\begin{eqnarray*}
\sum_{m=1}^{4k+2}\chi_{k}(m)B_1\left(\frac{m}{4k+2}\right)=-1.
\end{eqnarray*}
\end{proof}

\section{Other expansions and conjectures}
Other expansions for $F(q)$ frequently appear throughout the combinatorics literature. For example, we have \cite[A138265]{oeis}
\begin{eqnarray*}
F\left(\frac{1}{1+q} \right) = 1 + q + q^2 + 2q^3 + 5q^4 + 16q^5 + 61q^6 + 271q^7 + 1372q^8 + \cdots
\end{eqnarray*}
and \cite[A289312]{oeis}
\begin{eqnarray*} 
F\left(\frac{1-q}{1+q} \right) = 1 + 2q + 6q^2 + 26q^3 + 142q^4 + 946q^5 + 7446q^6 + 67658q^7 + 697118q^8 + \cdots.
\end{eqnarray*}
Using Theorem \ref{main}, we may deduce asymptotics for the coefficients of $F_f\left(\dfrac{1}{1+q}\right)$ and $F_f\left(\dfrac{1-q}{1+q}\right)$. Namely, if we write 
\begin{eqnarray*}
F_f\left(\dfrac{1}{1+q}\right)=:\sum_{n=0}^{\infty} g_f(n)q^n,
\end{eqnarray*}
then 
\begin{eqnarray} \label{fshift1}
g_f(n)\sim (-1)^\nu\left(\dfrac{M}{2\pi k_\nu}\right)^{2n+\nu+1}\dfrac{G_f^{(\nu)}(k_\nu)2^{2n+\nu} n! n^{\nu-\frac{1}{2}}}{b^n\sqrt{\pi M}}e^{-\frac{bk_\nu^2\pi^2}{2M^2}}.
\end{eqnarray}
This follows upon first noting
\begin{equation*} 
F_f\left(\dfrac{1}{1+q}\right) = F_f\left(1 - \dfrac{q}{1+q}\right) = \sum_{j=0}^{\infty} \xi_f(j)q^j \sum_{m=0}^{\infty} (-1)^m\binom{j+m-1}{m}q^m
\end{equation*}
and so
\begin{equation*}
g_f(n)=\sum_{\ell=0}^{n-1}(-1)^\ell\binom{n-1}{\ell}\xi_f(n-\ell),
\end{equation*}
then applying Theorem \ref{main}. Similarly, if
\begin{eqnarray*}
F_f\left(\dfrac{1-q}{1+q}\right)=:\sum_{n=0}^{\infty} h_f(n)q^n,
\end{eqnarray*}
then one can check
\begin{equation}\label{fshift2}
h_f(n)\sim (-1)^\nu\left(\dfrac{M}{2\pi k_v}\right)^{2n+\nu+1}\dfrac{G_f^{(\nu)}(k_\nu)2^{3n+\nu} n! n^{\nu-\frac{1}{2}}}{b^n\sqrt{\pi M}}.
\end{equation}

Asymptotics for the coefficients of $\mathscr{F}_t(q)$, $X_m^{(\ell)}(q)$ and $\mathscr{G}_k(q)$ with $q$ replaced by $\frac{1}{1+q}$ or $\frac{1-q}{1+q}$ now follow readily from (\ref{fshift1}), (\ref{fshift2}) and Corollaries \ref{genFishasymp1},  \ref{genFishasymp2} and \ref{genG}. Thus, all but finitely many coefficients are positive for $\mathscr{F}_t(q)$, $X_m^{(\ell)}(q)$ and $\mathscr{G}_k(q)$ where $q$ is replaced by $\frac{1}{1+q}$ or $\frac{1-q}{1+q}$. Interestingly, it appears numerically that more is true. Some supporting data is given below in Tables \ref{Ft2}--\ref{Gk}. 

\vspace{.05in}

\begin{table}[h]
{\tabulinesep=1mm
\begin{tabu}{|c|l|} \hline
$t=1$ & $1 + q + q^2 + 2 q^3 + 5 q^4 + 16 q^5 + 61 q^6 + 271 q^7 + 1372 q^8 + 
 7795 q^9 + 49093 q^{10}+\dotsc$ \\ \hline
$t=2$ & $1 + 3 q + 8 q^2 + 31 q^3 + 160 q^4 + 1029 q^5 + 7910 q^6 + 70658 q^7 + 718687 q^8 + \dotsc$ \\ \hline
$t=3$ & $1 + 7 q + 42 q^2 + 329 q^3 + 3395 q^4 + 43638 q^5 + 670663 q^6 + 11980513 q^7 + \dotsc$ \\ \hline
$t=4$ & $1 + 15 q + 190 q^2 + 3005 q^3 + 61885 q^4 + 1587420 q^5 + 48722721 q^6 + 1739070735 q^7 + \dotsc$  \\ \hline
$t=5$ & $ 1 + 31 q + 806 q^2 + 25637 q^3 + 1054465 q^4 + 54008696 q^5 + 3311724885 q^6 + \dotsc$\\ \hline
\end{tabu}} 
\caption{\label{Ft2} Coefficients for $\mathscr{F}_t(\frac{1}{1+q})$ for $1 \leq t \leq 5$}
\end{table}

\vspace{.05in}

\begin{table}[h]
{\tabulinesep=1mm
\begin{tabu}{|c|l|} \hline
$t=1$ & $1 + 2 q + 6 q^2 + 26 q^3 + 142 q^4 + 946 q^5 + 7446 q^6 + 67658 q^7 + 
 697118 q^8 + 8031586 q^9 + \dotsc$ \\ \hline
$t=2$ & $1 + 6q + 38 q^2 + 318 q^3 + 3406 q^4 + 44790 q^5 + 699126 q^6 + 12630702 q^7 + \dotsc$ \\ \hline
$t=3$ & $1 + 14 q + 182 q^2 + 2982 q^3 + 62734 q^4 + 1630174 q^5 + 50474886 q^6 + 1813113398 q^7 + \dotsc$ \\ \hline
$t=4$ & $1 + 30 q + 790 q^2 + 25590 q^3 + 1064590 q^4 + 54905390 q^5 + 3382387174 q^6 + \dotsc$  \\ \hline
$t=5$ & $1 + 62 q + 3286 q^2 + 211606 q^3 + 17496462 q^4 + 1797007566 q^5 + 220762565542 q^6 + \dotsc$\\ \hline
\end{tabu}} 
\caption{\label{Ft3} Coefficients for $\mathscr{F}_t(\frac{1-q}{1+q})$ for $1 \leq t \leq 5$}
\end{table}

\vspace{.05in}

\begin{table}[h]
{\tabulinesep=1mm
\begin{tabu}{|c|l|} \hline
$\ell=0$ & $ 1 + 5 q + 25 q^2 + 180 q^3 + 1725 q^4 + 20538 q^5 + 291571 q^6 + 4801844 q^7 + \dotsc$ \\ \hline
$\ell=1$ & $ 2 + 9 q + 45 q^2 + 330 q^3 + 3195 q^4 + 38286 q^5 + 545949 q^6 + 9020385 q^7 + \dotsc$ \\ \hline
$\ell=2$ & $ 3 + 12 q + 60 q^2 + 446 q^3 + 4350 q^4 + 52374 q^5 + 749294 q^6 + 12410001 q^7 + \dotsc$  \\ \hline
$\ell=3$ & $ 4 + 14 q + 70 q^2 + 525 q^3 + 5145 q^4 + 62139 q^5 + 890925 q^6 + 14779290 q^7 + \dotsc$\\ \hline
$\ell=4$ & $ 5 + 15 q + 75 q^2 + 565 q^3 + 5550 q^4 + 67134 q^5 + 963578 q^6 + 15997212 q^7+ \dotsc$\\ \hline
\end{tabu}} 
\caption{\label{Xml2} Coefficients for $X_5^{(\ell)}(\frac{1}{1+q})$ for $0 \leq  \ell \leq 4$}
\end{table}

\vspace{.05in}

\begin{table}[h]
{\tabulinesep=1mm
\begin{tabu}{|c|l|} \hline
$\ell=0$ & $1 + 10 q + 110 q^2 + 1650 q^3 + 32230 q^4 + 776666 q^5 + 22237534 q^6 + 737031746 q^7 + \dotsc$ \\ \hline
$\ell=1$ & $2 + 18 q + 198 q^2 + 3018 q^3 + 59598 q^4 + 1446210 q^5 + 41605014 q^6 + 1383694074 q^7 + \dotsc$ \\ \hline
$\ell=2$ & $3 + 24 q + 264 q^2 + 4072 q^3 + 81048 q^4 + 1976760 q^5 + 57067560 q^6 + 1902795528 q^7 + \dotsc$  \\ \hline
$\ell=3$ & $4 + 28 q + 308 q^2 + 4788 q^3 + 95788 q^4 + 2344076 q^5 + 67828068 q^6 + 2265402148 q^7 + \dotsc$\\ \hline
$\ell=4$ & $5 + 30 q + 330 q^2 + 5150 q^3 + 103290 q^4 + 2531838 q^5 + 73345162 q^6 + 2451727038 q^7 + \dotsc$\\ \hline
\end{tabu}} 
\caption{\label{Xml3} Coefficients for $X_5^{(\ell)}(\frac{1-q}{1+q})$ for $0 \leq \ell \leq 4$}
\end{table}

\vspace{.05in}

\begin{table}[h]
{\tabulinesep=1mm
\begin{tabu}{|c|l|} \hline
$k=1$ & $ 1 + q + 3 q^2 + 11 q^3 + 50 q^4 + 280 q^5 + 1892 q^6 + 15052 q^7 + 
 137957 q^8 + \dotsc$ \\ \hline
$k=2$ & $ 1 + 2 q + 8 q^2 + 42 q^3 + 293 q^4 + 2630 q^5 + 29054 q^6 + 
 380894 q^7 + 5773064 q^8 + \dotsc$ \\ \hline
$k=3$ & $ 1 + 3 q + 15 q^2 + 103 q^3 + 977 q^4 + 12137 q^5 + 186601 q^6 + 
 3411009 q^7 + 72158001 q^8+ \dotsc$  \\ \hline
$k=4$ & $ 1 + 4 q + 24 q^2 + 204 q^3 + 2454 q^4 + 39000 q^5 + 768720 q^6 + 
 18028512 q^7 + \dotsc$\\ \hline
$k=5$ & $ 1 + 5 q + 35 q^2 + 355 q^3 + 5180 q^4 + 100346 q^5 + 2413318 q^6 + 
 69085190 q^7 + \dotsc$\\ \hline
\end{tabu}} 
\caption{\label{Gk1} Coefficients for $\mathscr{G}_k(\frac{1}{1+q})$ for $1 \leq  k \leq 5$}
\end{table}

\vspace{.05in}

\begin{table}[h]
{\tabulinesep=1mm
\begin{tabu}{|c|l|} \hline
$k=1$ & $ 1 + 2 q + 6 q^2 + 34 q^3 + 278 q^4 + 2978 q^5 + 39302 q^6 + 
 615554 q^7 + 11151446 q^8 + \dotsc$ \\ \hline
$k=2$ & $ 1 + 4 q + 20 q^2 + 180 q^3 + 2420 q^4 + 42916 q^5 + 940244 q^6 + 
 24478804 q^7 + \dotsc$ \\ \hline
$k=3$ & $ 1 + 6 q + 42 q^2 + 518 q^3 + 9674 q^4 + 239302 q^5 + 7323946 q^6 + 
 266553414 q^7 + \dotsc$  \\ \hline
$k=4$ & $ 1 + 8 q + 72 q^2 + 1128 q^3 + 26952 q^4 + 855240 q^5 + 33608136 q^6 + 
 1571210280 q^7 + \dotsc$\\ \hline
$k=5$ & $ 1 + 10 q + 110 q^2 + 2090 q^3 + 60830 q^4 + 2355562 q^5 + 
 113032942 q^6 + 6454755274 q^7 + \dotsc$\\ \hline
\end{tabu}} 
\caption{\label{Gk} Coefficients for $\mathscr{G}_k\left(\frac{1-q}{1+q}\right)$ for $1 \leq  k \leq 5$}
\end{table}
\FloatBarrier

Based on the evidence in Remarks \ref{ev1} and \ref{ev2}, the computations in Remark \ref{checklm} and Tables \ref{Ft2}--\ref{Gk}, we make the following

\begin{conjecture} We have 
\begin{enumerate}
\item the coefficients of  $\mathscr{F}_t(1-q)$, $\mathscr{F}_t(\frac{1}{1+q})$ and $\mathscr{F}_t(\frac{1-q}{1+q})$ are positive for all $t \geq 1$.
\item the coefficients of $X_m^{(\ell)}(1-q)$, $X_m^{(\ell)}(\frac{1}{1+q})$ and $X_m^{(\ell)}(\frac{1-q}{1+q})$ are positive for all $m \in \mathbb{N}$ and $0 \leq \ell \leq m-1$.
\item the coefficients of $\mathscr{G}_k(\frac{1}{1+q})$ and    $\mathscr{G}_k (\frac{1-q}{1+q})$ are positive for all $k \geq 1$.
\end{enumerate}
\end{conjecture}

\section*{Acknowledgements}
The first author was an institute postdoctoral fellow at IIT Gandhinagar under the project IP/IITGN/MATH/AD/2122/15. He sincerely thanks the institute for the support. The second author was partially supported by SERB MATRICS grant MTR/2022/000659. The third author was supported by the Basic Science Research Program through the National Research Foundation of Korea (NRF) funded by the Ministry of Science and ICT (NRF-2019R1F1A1043415). The fourth author was partially supported by Enterprise Ireland CS20212030. The fourth author would also like to thank the Max-Planck-Institut f\"ur Mathematik for their support and hospitality during the completion of this paper. Finally, the authors thank Jeremy Lovejoy and the referee for helpful comments and suggestions.

\end{document}